\documentclass[12pt,oneside]{amsart}

\usepackage{amssymb,latexsym}

\usepackage{hyperref}
\hypersetup{
  colorlinks   = true, %Colours links instead of ugly boxes
  urlcolor     = blue, %Colour for external hyperlinks
  linkcolor    = blue, %Colour of internal links
  citecolor   = red %Colour of citations
}
\usepackage{anysize}

\usepackage{pgfplots}
\usepackage{mathrsfs}
\usetikzlibrary{arrows}
\usetikzlibrary{calc}
\usetikzlibrary{intersections}
\usetikzlibrary{through}

\usepackage{mathtools}

\newtheorem{theorem}{Theorem}[section]

\newtheorem{lemma}[theorem]{Lemma}
\newtheorem{claim}[theorem]{Claim}
\newtheorem{conjecture}[theorem]{Conjecture}
\newtheorem{corollary}[theorem]{Corollary}

\theoremstyle{definition}

\theoremstyle{remark}
\newtheorem{remark}[theorem]{Remark}

\numberwithin{equation}{section}

\DeclareMathOperator{\area}{area}

\newcommand{\Qmod}[2]{\widetilde{Q}_{#1,#2}}

\renewcommand{\epsilon}{\varepsilon}
\renewcommand{\phi}{\varphi}
\renewcommand{\kappa}{\varkappa}

\begin{document}

\title{Extensions of polynomial plank covering theorems}

\author{Alexey Glazyrin{$^\spadesuit$}}
\author{Roman Karasev{$^\clubsuit$}}
\author{Alexandr Polyanskii{$^\diamondsuit$}}

\thanks{{$^\spadesuit$} Partially supported by the NSF grant DMS-2054536}
\thanks{{$^\diamondsuit$} Partially supported by the Young Russian Mathematics award} 

\address{Alexey Glazyrin, School of Mathematical \& Statistical Sciences, The University of Texas Rio Grande Valley, Brownsville, TX 78520, USA}
\email{alexey.glazyrin@utrgv.edu}

\address{Roman Karasev, Institute for Information Transmission Problems RAS, Bolshoy Karetny per. 19, Moscow, Russia 127994}
\email{r\_n\_karasev@mail.ru}
\urladdr{http://www.rkarasev.ru/en/}

\address{Alexander Polyanskii, Department of Mathematics, Emory University, Atlanta, Georgia, 30322, US}
\email{alexander.polyanskii@gmail.com}


\begin{abstract}
We prove a complex polynomial plank covering theorem for not necessarily homogeneous polynomials. As the consequence of this result, we extend the complex plank theorem of Ball to the case of planks that are not necessarily centrally symmetric and not necessarily round. We also prove a weaker version of the spherical polynomial plank covering conjecture for planks of different widths.
\end{abstract}

\maketitle

\section{Introduction}

In 1931, Tarski \cite{tarski1931} formulated several questions about the degree of equivalence of polygons. Answering one of those questions, Moese \cite{moese1932} came up with the first version of the plank covering theorem for a planar disk. In 1932, Tarski \cite{tarski1932} mentioned a more general version of the plank problem about covering a planar convex body of given width. In 1950, Bang \cite{bang1950,bang1951} extended the plank covering theorem to an arbitrary dimension and to all bodies of given width using an ingenious combinatorial argument. In particular, he showed that a convex body of width 1 in the Euclidean space $\mathbb{R}^d$ can be covered by planks, that is, sets between two parallel hyperplanes, only if the sum of widths (distances between hyperplanes) of all planks is at least 1. A similar problem was posed by Fejes T\'oth \cite{toth1973} for covering the sphere. He conjectured that the sum of spherical widths of zones, that is, centrally symmetric parts of the sphere covered by one plank, is at least $\pi$ whenever zones cover the whole sphere. In \cite{polyajiang2017} Jiang and Polyanskii used the argument of Bang and other combinatorial ideas to prove this conjecture.

For both problems mentioned above, a plank of Euclidean width $2\delta$ may be defined as the set of points $x\in\mathbb{R}^d$ satisfying $|(x-y)\cdot u |\leq \delta$, where $y$ is another point in $\mathbb{R}^d$ and $u$ is a unit vector. This definition of a plank can be transferred verbatim to the complex case, where by a (round) plank we mean the set of points $x\in\mathbb{C}^d$ satisfying $|\langle x-y, u\rangle |\leq \delta$, where $y$ is another point in $\mathbb{C}^d$, $u$ is a unit complex vector, and $\langle\cdot,\cdot\rangle$ is the Hermitian product. Note that geometrically a plank in a complex vector space is not a set between two hyperplanes anymore but rather a circular cylinder. Ball \cite{ball2001} studied the case of centrally symmetric complex planks and found a necessary condition for them to cover the unit complex sphere in $\mathbb{C}^d$.

Recently, based on the ideas of Ball from \cite{ball2001}, Ortega-Moreno found a new proof of the Fejes T\'oth zone conjecture in the case of zones of equal widths \cite{ortega2021optimal} and simplified the proof of Ball in the complex case \cite{ortegamoreno2021}. Zhao \cite{zhao2021} simplified the proof of the equal width case of the zone conjecture even further. Our results in \cite{glazyrin2021} and in this paper develop the ideas of Ball and Ortega-Moreno, further simplifying the argument and extending it to new versions of the problem.

In the papers \cite{ball2001, ortega2021optimal, zhao2021,ortegamoreno2021}, the key idea is to consider the polynomial $\langle x, u_1 \rangle \ldots \langle x, u_N \rangle$ defined by linear parts in the definition of a plank (there are no shifting vectors $y$ because planks are centrally symmetric in all these works) and estimate the distance from the maximal absolute value of this polynomial on the unit sphere to its zero set. Although resulting sets do not resemble physical planks anymore, it seems natural to extend the notion of planks to incorporate polynomials.

By \textit{a polynomial plank} we mean a set of points defined by $\{x\ |\ \text{dist}(x,Z(P))\leq \delta\}$, given a polynomial $P$ and $\delta>0$, where $Z(P)$ is the zero set of $P$ in the ambient space. This definition is somewhat vague because we would like to cover different scenarios for spaces ($\mathbb{R}^d$, $\mathbb{C}^d$, $S^d$, $\mathbb{RP}^d$, $\mathbb{CP}^d$) and suitable metrics in them. The general question is to determine necessary conditions on degrees of polynomials $P_1,\ldots,P_N$ and corresponding widths of planks determined by $\delta_1,\ldots, \delta_N$ such that polynomial planks cover the whole space or the unit ball in the space. Some of our results (e.g. Corollary~\ref{corollary:ortega-moreno3}) can be phrased in terms of polynomial planks but we prefer to give more explicit statements.

\begin{remark}
There is a series of papers where the authors study polynomial versions of the plank problem, for example, \cite{pinasco2012,kavkim2012,cpt2017}. The main approach in these papers is quite different from ours, albeit the same in the case of homogeneous linear polynomials. Roughly speaking, in these papers the authors are interested in bounding the value $|P(x)|$ rather than bounding the distance from $x$ to the zero set of $P$ as in our case.
\end{remark}

In paper \cite{glazyrin2021}, we essentially answered the question above in the case of the complex projective space, that is, for homogeneous complex polynomials. More interestingly, for the real sphere we extended the approach of Ortega-Moreno \cite{ortega2021optimal}, later simplified by Zhao \cite{zhao2021}, to nonhomogeneous polynomials, thus answering the question for polynomial planks of equal widths, and used this extension to prove the generalized version of the Fejes T\'oth zone conjecture by showing that the total spherical width of spherical segments (parts of the sphere covered by a not necessarily centrally symmetric plank) is at least $\pi$. The latter result on spherical segments essentially used the polynomial technique and does not seem to follow from the methods in~\cite{polyajiang2017}. Finally, we were able to prove the polynomial version of the Bang theorem for the unit ball in the Euclidean space and for polynomial planks of the same width.

The main result of this paper is a nonhomogeneous extension of the complex polynomial plank problem. This result has various consequences including the generalization of the complex plank theorem by Ball to planks that are not necessarily round and not necessarily centrally symmetric. In the spherical case, we prove the first result for polynomial planks of different widths, though the constant for total width in our theorem is weaker than the conjectured one.

\subsection{Avoiding zeroes of complex not necessarily homogeneous polynomial}

Recall the result from \cite{ortegamoreno2021,glazyrin2021} that we are aimed to modify.

\begin{theorem}[Theorem~1.10 of \cite{glazyrin2021}, based on the ideas of \cite{ball2001,ortegamoreno2021}]
\label{theorem:ortega-moreno2o}
Assume that $P_1,\ldots, P_N\in \mathbb C[z_1, \ldots, z_d]$ are nonzero homogeneous polynomials and $\delta_1,\ldots, \delta_N > 0$ are such that
\[
\sum_{k=1}^N \delta_k^2 \deg P_k \le 1.
\]
Then the point of maximum of the absolute value of $P_1^{\delta_1^2}\cdots P_N^{\delta_N^2}$ on the unit sphere $S^{2d-1}\subset\mathbb C^d$ is, for every $k$, at angular distance at least $\arcsin\delta_k$ from the intersection of the zero set of $P_k$ with $S^{2d-1}$.
\end{theorem}

We are interested in dropping the assumption that polynomials are homogeneous from this theorem. This cannot be done directly due to the example in \cite[Remark~1.12]{glazyrin2021}. For the modification to succeed, something has to be changed in the statement. The first change we do is passing from the sphere to the ball. The second change, similar to \cite[Theorem~1.5]{glazyrin2021}, is introducing an additional radially symmetric multiplier before taking the maximum. Curiously, the multiplier in this case is simpler and the proof also simplifies the proof of Theorem \ref{theorem:ortega-moreno2o} (although in Appendix~\ref{section:appendix} we present a longer proof following the strategy of the proof of \cite[Theorem~1.5]{glazyrin2021}).

\begin{theorem}
\label{theorem:ortega-moreno4}
Let $P_1, \dots, P_N\in \mathbb C[z_1,\dots,z_d]$ be nonzero polynomials and $\delta_1,\ldots, \delta_N > 0$ are such that
\[
\sum_{j=1}^N \delta_j^2 \deg P_j \leq R^2.
\]
Then any point of maximum of the absolute value of $e^{-|z|^2/2}P_1^{\delta_1^2}\cdots P_N^{\delta_N^2}$ on the ball $B^{2d}(R)\subset\mathbb C^d$ $($of radius $R)$ is, for every $k$, at Euclidean distance $($in $\mathbb C^d)$ at least $\delta_k$ from the zero set of $P_k$.
\end{theorem}

\begin{corollary}
\label{corollary:ortega-moreno3}
Assume that $P_1,\ldots, P_N\in \mathbb C[z_1, \ldots, z_d]$ are nonzero polynomials and $\delta_1,\ldots, \delta_N > 0$ are such that
\[
\sum_{k=1}^N \delta_k^2 \deg P_k \le R^2.
\]
Then there exists a point in the ball $B^{2d}(R)\subset\mathbb C^d$ at Euclidean distance at least $\delta_k$ from the zero set of $P_k$, for all $k$.
\end{corollary}

This result implies in the usual fashion an analogue of the result of \cite{ball2001} for coverings by not necessarily centrally symmetric planks.

\begin{corollary}
\label{corollary:ball-nonhomogeneous}
Assume that $\delta_1,\ldots, \delta_N > 0$ and
\[
\sum_{k=1}^N \delta_k^2 < R^2.
\]
Then the ball $B^{2d}(R)\subset\mathbb C^d$ cannot be covered by a union of $N$ cylinders, the $k^{\text{th}}$ cylinder being the $\delta_k$-neighborhood of a complex affine hyperplane in $\mathbb C^d$.
\end{corollary}

The case $n=1$ of this result reads: \emph{If a disk in the plane is covered by a finite set of disks then the sum of squared radii of the covering disks is greater or equal to the squared radius of the covered disk.} It is remarkable that the argument below proves this without using the notion of area, from which this statement obviously follows.

In view of the symplectic capacity subadditivity conjecture \cite[Conjecture~4.2]{akp2014} the following slightly generalized corollary may be interesting. Note that in \cite{akp2014} it is incorrectly mentioned that Corollary~\ref{corollary:ball-nonhomogeneous} is proved in \cite{ball2001}, while in fact only the case of centrally symmetric planks is proved there.

\begin{corollary}
\label{corollary:ball-non-round}
If the ball $B^{2d}(R)\subset\mathbb C^d$ is covered by unitary planks then the sum of cross-section areas of the planks is at least $\pi R^2$. Here a unitary plank $P$ is a unitary image of $K\times\mathbb C^{d-1}\subset \mathbb C^d$ where $K\subseteq \mathbb C$ is measurable; the area of $K$ is the cross-section area of $P$. 
\end{corollary}

Also Corollary \ref{corollary:ball-nonhomogeneous} immediately implies the following nonhomogeneous generalization of the result of Arias-de-Reyna on the complex linear polarization constant \cite{arias1998}.

\begin{corollary}
\label{corollary:nonhomogeneous-polarization}
For any unit vectors $u_1,\ldots,u_d\in\mathbb{C}^d$ and any vectors $y_1,\ldots,y_d\in\mathbb{C}^d$, there exists a unit vector $x\in\mathbb{C}^d$ such that
$$|\langle x-y_1, u_1\rangle|\ldots |\langle x-y_d, u_d\rangle|\geq n^{-n/2}.$$
\end{corollary}

Theorem~\ref{theorem:ortega-moreno2o} (or Theorem~\ref{theorem:ortega-moreno4}) has another consequence.

\begin{corollary}
\label{corollary:cp-covering-kd}
For any $d\ge 2$ unit vectors $v_1,\ldots,v_d\in \mathbb C^d$ there exists another unit vector $q\in S^{2d-1}$ such that its Euclidean distance to the linear span of any $k$ of the $v_i$ is at least $\sqrt\frac{d-k}{d}$.
\end{corollary}

A similar result in the real case is covered by \cite[Theorem~1.3]{ballprodromou2009}, where the bound is given only for $k=1$ but for an arbitrary number $n$ of vectors. Our method also allows us to have an estimate for $n\ge d$ vectors.

\begin{corollary}
\label{corollary:cp-covering-n}
For any $n\ge d\ge 2$ unit vectors $v_1,\ldots,v_n\in \mathbb C^d$ there exists another unit vector $q\in S^{2d-1}$ such that its Euclidean distance to any of the $v_i$ is at least $\sqrt\frac{d-1}{n}$.
\end{corollary}

As $n\to\infty$, this estimate is asymptotically worse than the obvious volumetric estimate of order $n^{-1/(2d-1)}$. But this estimate is tight for $d=n$ (the vectors forming an orthogonal basis) and may be useful for small $n$.

\subsection{Avoiding zeroes of real polynomials with different distances}

In \cite{glazyrin2021}, the following result for spherical coverings by polynomial planks led to the proof of the generalized zone conjecture of Fejes T\'oth.

\begin{theorem}[Theorem 1.1 of \cite{glazyrin2021}, based on the ideas of \cite{ortega2021optimal,zhao2021}]
\label{theorem:zhao-inequality}
If a polynomial $P\in \mathbb R[x_1, \ldots, x_d]$ of degree $n$ has a nonzero restriction to the unit sphere $S^{d-1}\subset\mathbb R^d$ and attains its maximal absolute value on $S^{d-1}$ at a point $p$ then $p$ is at angular distance at least $\frac{\pi}{2n}$ from the intersection of the zero set of $P$ with~$S^{d-1}$.
\end{theorem}

This result implied the polynomial plank covering theorem for the Euclidean ball.

\begin{corollary}\cite[Corollary 1.7]{glazyrin2021}
\label{corollary:ball-strong}
For every nonzero polynomial $P\in \mathbb R[x_1, \ldots, x_d]$ of degree $n$, there exists a point of $B^d\subset\mathbb R^d$ at distance at least $\frac{1}{n}$ from the zero set of the polynomial $P$.
\end{corollary}

We conjectured in \cite{glazyrin2021} a version of this result with different distances to different sets of zeros.

\begin{conjecture}\cite[Conjecture 1.8]{glazyrin2021}
\label{conjecture:polynomial-planks}
Assume that $P_1,\ldots, P_N\in \mathbb R[x_1, \ldots, x_d]$ are nonzero polynomials and $\delta_1,\ldots, \delta_N > 0$ are such that
\[
\sum_{k=1}^N \delta_k \deg P_k \le 1.
\]
Then there exists a point $p\in B^d\subset\mathbb R^d$ such that, for every $k=1,\ldots, N$, the point $p$ is at distance at least $\delta_k$ from the zero set of $P_k$.
\end{conjecture}

By essentially repeating the proof of \cite[Theorem~1.5]{glazyrin2021} and \cite[Corollary~1.7]{glazyrin2021}, one can show that Conjecture \ref{conjecture:polynomial-planks} follows from the corresponding conjecture about the sphere. Although implied, it was not formulated explicitly in \cite{glazyrin2021}, so we state it here.

\begin{conjecture}
\label{conjecture:spherical-polynomial-planks}
Assume that polynomials $P_1,\ldots, P_N\in \mathbb R[x_1, \ldots, x_d]$ have nonzero restrictions to the unit sphere $S^{d-1}\subset\mathbb R^d$ and $\delta_1,\ldots, \delta_N > 0$ are such that
\[
\sum_{k=1}^N \delta_k \deg P_k \le \frac {\pi} 2.
\]
Then there exists a point $p\in S^{d-1}$ such that, for every $k=1,\ldots, N$, the point $p$ at angular distance at least $\delta_k$ from the intersection of the zero set of $P_k$ with~$S^{d-1}$.
\end{conjecture}

% May be not needed when integrated into the complex non-homogeneous text
%\begin{remark}
%In \cite{glazyrin2021}, we also proved the complex homogeneous analog of this conjecture.
%\end{remark}

%
%\begin{theorem}
%\label{theorem:ortega-moreno2}
%Assume that $P_1,\ldots, P_N\in \mathbb C[z_1, \ldots, z_d]$ are nonzero homogeneous polynomials and $\delta_1,\ldots, \delta_N > 0$ are such that
%\[
%\sum_{k=1}^N \delta_k^2 \deg P_k \le 1.
%\]
%Then the point of maximum of the absolute value of $P_1^{\delta_1^2}\cdots P_N^{\delta_N^2}$ on the unit sphere $S^{2d-1}\subset\mathbb C^d$ is, for every $k$, at angular distance at least $\arcsin\delta_k$ from the intersection of the zero set of $P_k$ with $S^{2d-1}$.
%\end{theorem}

The general machinery in the proofs of the statements above is to estimate the distance between the maximum absolute value of a certain function and the zero of this function with a particular order. This method does not seem to work directly due to the example in \cite[Remark~1.13]{glazyrin2021}. However, it leads to a weaker version of Conjecture~\ref{conjecture:spherical-polynomial-planks} with a smaller constant.

\begin{theorem}
\label{theorem:spherical-polynomial-planks}
Assume that polynomials $P_1,\ldots, P_N\in \mathbb R[x_1, \ldots, x_d]$ have nonzero restrictions to the unit sphere $S^{d-1}\subset\mathbb R^d$ and $\delta_1,\ldots, \delta_N > 0$ are such that
\[
\sum_{k=1}^N \delta_k \deg P_k \le \frac 1 e.
\]
There exists a point $p\in S^{d-1}$ of the maximum of the absolute value of $P_1^{\delta_1}\cdots P_N^{\delta_N}$ on the unit sphere $S^{d-1}$ such that for every $k=1,\ldots, N$, $p$ is at angular distance at least $\delta_k$ from the intersection of the zero set of $P_k$ with~$S^{d-1}$.
\end{theorem}

We restate \cite[Conjecture~1.17]{glazyrin2021} here as an example of a question so far resisting the approach of Bang \cite{bang1951} (including its versions in \cite{polyajiang2017,polyanskii2021cap}) and the polynomial approach we currently consider.

\begin{conjecture}[Conjecture~1.17 in \cite{glazyrin2021}]
If $d\ge 4$ and the unit sphere $S^{d-1} \subset \mathbb R^d$ is covered by a finite number of real planks then the sum of their widths is at least $2$.
\end{conjecture}

\subsection*{Acknowledgments} The authors thank Arseniy Akopyan for useful discussions and the anonymous referee for useful remarks.

\section{Proof of Theorem \ref{theorem:ortega-moreno4} on complex polynomials}

Let us redefine the radius by the equality
\[
R^2 = \sum_{j=1}^N \delta_j^2 \deg P_j.
\]
The argument below will show that all points of maximum of the expression in the statement of the theorem belong to the ball of this redefined, possibly smaller radius. Hence Theorem~\ref{theorem:ortega-moreno4} also holds true for any larger radius.

Let us show that the global maximum of the expression
\[
F(z) = e^{-|z|^2/2} \left| P_1^{\delta_1^2}\cdots P_N^{\delta_N^2} \right|
\]
as a function of $z\in\mathbb C^d$ is attained at some point $z$ with $|z| \le R.$ After considering the complex line through the point of maximum and the origin, the question gets reduced to the one-dimensional case. Observe that the expression 
\[
e^{-|z|^2/2} |z|^{R^2}
\]
decreases when $|z|^2\ge R^2$ as a function of $|z|$. The remaining factor
\[
|z|^{-R^2} \left| P_1^{\delta_1^2}\cdots P_N^{\delta_N^2} \right|
\]
is subharmonic in the domain $|z|^2\ge R^2$ including $z=\infty$, where it has finite limit. Then the maximum principle for the latter factor and monotonicity of the former factor exclude global maxima of the original expression with $|z|^2 > R^2$.

Now consider a point of maximum $z_M$ of $F(z)$ as a function of $z\in\mathbb C^d$. Let $z_0$ be a zero of $P_k$. Pass to the one-dimensional line through $z_0$ and $z_M$, choose the coordinate $w$ on the line so that $w=0$ corresponds to $z_0$ and 
\[
w = a := |z_M-z_0| \in \mathbb R_+
\] 
corresponds to the maximum. Our function up to a constant factor then becomes
\begin{multline*}
f(w) = e^{|z_0|^2/2} F( z_0 + w/a (z_M - z_0) ) = \\
= e^{-|w|^2/2 - \Re \bar z_0\cdot (z_M-z_0) w/a} \left| P_1^{\delta_1^2}( z_0 + w/a (z_M - z_0) )\cdots P_N^{\delta_N^2}( z_0 + w/a (z_M - z_0) ) \right|.   
\end{multline*}

This can be split into two factors, 
\[
g(w) = e^{-|w|^2/2} |w|^{\delta_k^2}
\]
and
\begin{multline*}
h(w) = \left| e^{\bar z_0\cdot (z_M-z_0) w/a} \right| \cdot\\
\cdot \left| P_1^{\delta_1^2}( z_0 + w/a (z_M - z_0) )\cdots \frac{P_k^{\delta_k^2}( z_0 + w/a (z_M - z_0) )}{w^{\delta_k^2}} \cdots P_N^{\delta_N^2}( z_0 + w/a (z_M - z_0) ) \right|.
\end{multline*}
The first factor increases when $|w|$ is in the range $[0, \delta_k]$, the second one is subharmonic in the disc $|w|\le \delta_k$ as the absolute value of an analytic function. Hence the maximum of the product in the disc $|w|\le \delta_k$ may only be on its boundary. This proves that $a\ge \delta_k$.

\begin{remark}[Observed by Arseniy Akopyan]
If all polynomials $P_i$ are homogeneous then the restriction of $F(z)$ to any one-dimensional linear subspace of $\mathbb C^d$ is proportional to $e^{-|z|^2/2} |z|^{\sum_{i=1}^N \delta_i^2\deg P_i} = e^{-|z|^2/2} |z|^{R^2}$, whose maximum is attained at $|z|=R$. Hence the global maximum of $F$ in this case lies on the sphere of radius $R$, which implies that the above argument essentially proves Theorem~\ref{theorem:ortega-moreno2o} on homogeneous polynomials on the sphere.
\end{remark}

\begin{remark}
\label{remarl:k-wise-intersection-c}
The above argument implies that under the assumptions of Theorem~\ref{theorem:ortega-moreno4}, for any subset of indices $1\le i_1<\dots< i_k\le N$, the point of maximum of $F$ is at Euclidean distance at least
\[
\sqrt{\delta_{i_1}^2 + \dots + \delta_{i_k}^2}
\]
from the solutions of the system of equations
\[
P_{i_1}(z) = \dots = P_{i_k}(z) = 0.
\]
Indeed, when restricted to the line through a solution of these equations (corresponding to $w=0$) and the point of maximum, the function $F$ factorizes into $e^{-|w|^2/2} |w|^{\delta_{i_1}^2 + \dots + \delta_{i_k}^2}$ and something subharmonic. Note that the bound is tight in the case when polynomials $P_i$ are the coordinates in $\mathbb C^N$. We use this observation to prove Corollaries \ref{corollary:cp-covering-kd} and \ref{corollary:cp-covering-n}. 
\end{remark}

\begin{proof}[Proof of Corollary~\ref{corollary:cp-covering-kd}]
Assume $v_i$ are linearly independent. Otherwise, we can take $q$ orthogonal to all of them.

Let $L_i$ be the complex linear form such that
\[
L_i(v_i)\neq 0,\quad \forall j\neq i\quad L_i(v_j) = 0.
\]
Let $q$ be a point of maximum of $|L_1\cdots L_d|$ on the unit sphere. 

Without loss of generality, we need to show that $q$ is at Euclidean distance at least $\sqrt{\frac{d-k}{d}}$ from the linear span of $v_1,\ldots, v_k$. Note that this linear span $Z$ is the common zero set of $L_{k+1}, \ldots, L_d$.

Assuming the contrary, restrict the product $L_1\cdots L_d$ to the two-dimensional linear subspace $V$ spanned by $q$ and a point $p\in Z$ at distance less than $\sqrt{\frac{d-k}{d}}$ from $q$. Then the product in question restricts to a homogeneous form factorized as $M^{d-k} N$, where $M$ is linear, $M(p)=0$, and $N$ has degree $k$. Apply Theorem~\ref{theorem:ortega-moreno2o} to $|M(z)^{\frac{d-k}{d}} N(z)^{\frac{1}{d}}|$ and its point of maximum $q$. It asserts that $q$ is at Euclidean distance at least $\sqrt{\frac{d-k}{d}}$ from the line spanned by $p$, thus leading to a contradiction.
\end{proof}

\begin{remark}
For $k=1$, the last step of the proof can be made without a reference to Theorem~\ref{theorem:ortega-moreno2o}. One can simply analyze the maximum of $u^{d-1}(v + \alpha u)$ on the unit sphere $\{|u|^2+|v|^2=1\}$ manually and show that $|u|\ge \sqrt{\frac{d-1}{d}}$ at the point of maximum.
\end{remark}

Although Corollary~\ref{corollary:cp-covering-kd} does look like a fact that should be already known, we were not able to locate it in the literature even for the case $k=1$. It is interesting that for $k=1$ there is a very short proof that is essentially based on the original argument of Bang. For completeness, we include this proof here.

\begin{proof}[Alternative proof of the particular case $k=1$ of Corollary~\ref{corollary:cp-covering-kd}]
As in the previous proof, we can assume all $v_i$ are linearly independent. Then we can take the dual basis $w_1,\ldots, w_d$, that is, the one satisfying $\langle v_i, w_j\rangle = \delta_{ij}$ for all $1\leq i,j \leq d$. Note that $|w_i|$ must be at least 1 for all $i$.

Now we take the random vector $u_f=f_1 w_1 + \ldots + f_d w_d$, where $f=\{f_i\}$ is a sequence of i.i.d. Steinhaus random variables (uniformly distributed over a unit complex circle). Then
$$\mathsf{E} |u_f|^2 = \mathsf{E} \left\langle \sum\limits_{i=1}^d f_i w_i, \sum\limits_{i=1}^d f_i w_i \right\rangle =$$ $$=\sum\limits_{i=1}^d \sum\limits_{i=1}^d \langle w_i, w_j\rangle \mathsf{E} f_i \overline{f_j} = \sum\limits_{i=1}^d \sum\limits_{i=1}^d \langle w_i, w_j\rangle \delta_{ij} = \sum\limits_{i=1}^d |w_i|^2 \geq d.$$

This means there is a choice of values for $f$ such that $|u_f|\geq \sqrt{d}$. Then $u_f/|u_f|$ is a suitable choice for $q$. Indeed, $|\langle v_i, u_f/|u_f|\rangle| = 1/|u_f| \leq 1/\sqrt{d}$ so the Euclidean distance from $u_f/|u_f|$ to a line defined by $v_i$ is at least $\sqrt{\frac{d-1}d}$ for all $1\leq i\leq d$.
\end{proof}

\begin{proof}[Proof of Corollary~\ref{corollary:cp-covering-n}]
We assume $v_i$ are in general position, that is, any $d$ of them are linearly independent. The general case follows from the generic one by passing to the limit and the usual compactness argument.

Consider all $\binom{n}{d-1}$ hyperplanes spanned by $(d-1)$-tuples of $v_i$. Let $P$ be the product of their respective linear forms and $q$ be a unit vector maximizing $|P|$. Consider one $v_i$ and pass to the two-dimensional linear span of $q$ and $v_i$. The restriction of the product $P$ to this subspace is a homogeneous polynomial of degree $\binom{n}{d-1}$ having zero of multiplicity $\binom{n-1}{d-2}$ at $v_i$, as this is the number of hyperplanes passing through $v_i$. Theorem~\ref{theorem:ortega-moreno2o} then implies that the Euclidean distance from $q$ to $v_i$ is at least 
\[
\sqrt{\frac{\binom{n-1}{d-2}}{\binom{n}{d-1}}} = \sqrt{\frac{d-1}{n}}.
\]
\end{proof}

\section{Proofs of the complex plank covering corollaries}

\begin{proof}[Proof of Corollary \ref{corollary:ortega-moreno3}]
Follows directly from Theorem~\ref{theorem:ortega-moreno4}.
\end{proof}

\begin{proof}[Proof of Corollary~\ref{corollary:ball-non-round}]
Assuming the contrary it is possible to put every cylinder into an open cylinder, corresponding to an inclusion of planar bodies $K\subseteq U$, so that the sum of cross-section areas is still strictly less than $\pi R^2$. From the open covering one can leave only finite number of cylinders using the compactness of a ball.

Using the Lebesgue covering lemma then it is possible to pass to the compact $C\subset U$ (complement of $\delta$-neighborhood of the complement of $U$) so that the ball is covered by these smaller closed cylinders.

After that, for any $\delta>0$ one may cover $C$ with a finite collection of discs such that the sum of areas of the discs is at most $\area (C) + \delta$. Indeed, one may first cover almost all of $C$ by a countable sequence of disjoint discs whose total area is less than $\area (C) + \delta/2$ by the Besicovitch covering theorem, then cover the remaining set of measure zero by a countable collection of discs of total measure less than $\delta/2$. Then the compactness of $C$ allows us to leave a finite collection of discs in the covering. 

Such a covering of $C$ by discs corresponds to a covering of the plank $C\times \mathbb C^{d-1}$ or its unitary image by round cylinders. Taking sufficiently small $\delta>0$ and doing the procedure for every plank in the covering, one then obtains a covering of the ball $B^{2d}(R)$ by a finite set of round cylinders with total cross-section area strictly less than $\pi R^2$. This contradicts Corollary~\ref{corollary:ball-nonhomogeneous}.
\end{proof}

\section{Proof of Theorem~\ref{theorem:spherical-polynomial-planks} on real polynomials}

For the proof of the theorem, we need the following lemma.

\begin{lemma}
\label{lemma-bernstein}
Let $Q$ be a trigonometric polynomial of degree $n$ with the root of order $k$ at 0. Let $Q(t_0)=\max\limits_{[0,2\pi]} |Q|$ for $t_0\in[0,2\pi]$. Then 
	
a$)$\footnote{The formula in Lemma~\ref{lemma-bernstein}(a) is corrected after the official publication. The proof is corrected accordingly.} $t_0\geq \left(\frac{n^k}{k!}\right)^{-\frac 1 k}$,
	
b$)$ $t_0\geq \frac k {en}$.
\end{lemma}

\begin{proof}
Let $\max\limits_{[0,2\pi]} |Q|=M$ and assume that $Q(t_0) = M$, with the case $Q(t_0)=-M$ being essentially the same. The Bernstein inequality reads 
\[
\|Q'\|_C \le \deg Q\cdot \|Q\|_C.
\]
Using it $k$ times, we obtain the bound
\[
Q^{(k)}(t) \leq n^k M. 
\]
It follows that 
\[
Q(t)\leq M \frac {n^k t^k} {k!}
\] 
for all $t\in[0,2\pi]$, because the first $k-1$ derivatives of both sides at $0$ are $0$ and the inequality above is precisely the one on their $k^{\text{th}}$ derivatives. Using this inequality for a point of maximum $t_0$, we get 
\[
M=Q(t_0)\leq M \frac{n^k}{k!} t_0^k,
\]
which implies 
\[t_0\geq \left(\frac{n^k}{k!}\right)^{-\frac 1 k}.
\]
	
For part b), we use part a) for the polynomial $Q^N$ and take $N\rightarrow\infty$. For simplicity, denote $nN$ by $L$ and $k/n$ by $\alpha$. Then, using Stirling's approximation formula,
\[
t_0\geq \lim\limits_{L\rightarrow\infty} \left(\frac {(\alpha L)!} {L^{\alpha L}}\right)^{\frac 1 {\alpha L}}=
\lim\limits_{L\rightarrow\infty} \frac{{\left( (\alpha L)!\right)}^{\frac 1 {\alpha L}}} {L}  =
= \lim\limits_{L\rightarrow\infty} \frac {\alpha L/e} {L} = \frac {\alpha} {e}.
\]
\end{proof}

\begin{proof}[Proof of Theorem \ref{theorem:spherical-polynomial-planks}]
First, we prove the theorem for rational $\delta_i$. Denote the least common multiple of the denominators of $\delta_i$ by $D$ and define $P={P_1^{D\delta_1}\cdots P_N^{D\delta_N}}$. This is a polynomial of degree $n=\sum_{k=1}^N D\delta_k \deg P_k \le \frac D e$. Let $p$ be a point of the maximum of the absolute value of $P$. Assume $q$ is the closest to $p$ point from the zero set of $P_k$. For the linear span of $p$ and $q$, we get that the restriction of $P$ to it has the maximal absolute value at $p$ and the root of order $D\delta_k$ at $q$. By Lemma \ref{lemma-bernstein}, the angular distance between them is at least $\frac {D\delta_k} {en}\geq \delta_k$.
	
The case of irrational $\delta_i$ follows by taking a limit for the sequence of rational values of $\delta_i$ and choosing the limit of a converging subsequence of corresponding points $p$.
\end{proof}

\begin{remark}
\label{remarl:k-wise-intersection-r}
Similar to Remark~\ref{remarl:k-wise-intersection-c}, the above argument implies that under the assumptions of Theorem~\ref{theorem:spherical-polynomial-planks}, for any subset of indices $1\le i_1<\dots< i_k\le N$, a point of maximum of $F = |P_1^{\delta_1}\cdots P_N^{\delta_N}|$ is at spherical distance at least
\[
\delta_{i_1} + \dots + \delta_{i_k}
\]
from the solutions of the system of equations
\[
P_{i_1}(x) = \dots = P_{i_k}(x) = 0
\]
on the sphere.
\end{remark}

\section{Appendix: Alternative proof of Theorem~\ref{theorem:ortega-moreno4} on complex polynomials}
\label{section:appendix}

It may seem somewhat mysterious that the factor $e^{-|z|^2/2}$ shows up in Theorem~\ref{theorem:ortega-moreno4}. The alternative proof sheds light on its origin.

In this proof we follow the approach to the proof of \cite[Theorem~1.5]{glazyrin2021} (in some sense going back to \cite{bognar1961}) of adding one more variable and passing to the sphere of one dimension higher. Note that the argument below proves a weaker version of the theorem, replacing ``any point of maximum'' by ``a point of maximum''.

In order for this plan to succeed we need the following version of Theorem~\ref{theorem:ortega-moreno2o}. Unlike in the original statement, here the radius of the sphere may be an arbitrary positive number and distances are measured in the ambient space $\mathbb C^d$. Since zero sets of homogeneous polynomials are cones, this version is equivalent to the original statement.

\begin{theorem}[Essentially Theorem~1.10 of \cite{glazyrin2021}]
\label{theorem:ortega-moreno2}
Assume that $P_1,\ldots, P_N\in \mathbb C[z_1, \ldots, z_d]$ are nonzero homogeneous polynomials and $\delta_1,\ldots, \delta_N > 0$ are such that
\[
\sum_{k=1}^N \delta_k^2 \deg P_k \le R^2.
\]
Then a point of maximum of the absolute value of $P_1^{\delta_1^2}\cdots P_N^{\delta_N^2}$ on the sphere $S^{2d-1}(R)\subset\mathbb C^d$ $($of radius $R)$ is, for every $k$, at Euclidean distance $($in $\mathbb C^d)$ at least $\delta_k$ from the zero set of $P_k$.
\end{theorem}

Now we proceed to the proof of Theorem~\ref{theorem:ortega-moreno4}. Let us homogenize the polynomials $P_1,\dots, P_N$ by replacing every monomial $z_1^{m_1}\dots z_d^{m_d}$ of $P_j$ with
\[
z_0^{\deg P_j - (m_1+\dots+m_d)} z_1^{m_1} \dots z_d^{m_d},
\]
and consider the polynomials $Q_1,\dots, Q_N\in \mathbb C[z_0,z_1,\dots, z_d]$ satisfying
\[
P_j(z_1,\dots, z_d)= Q_j(1,z_1,\dots, z_d),
\]
for all $j$.

Take a positive $\delta_0$ and further modify the polynomials depending on $\delta_0$
\[
\Qmod{j}{\delta_0}(z_0, z_1,\dots, z_d)=Q\Big(\frac{z_0}{\delta_0},z_1,\dots, z_d\Big).
\]
Additionally, let $\Qmod{0}{\delta_0}\in \mathbb C[z_0]$ be the polynomial defined by
\[
\Qmod{0}{\delta_0}(z_0)=\frac{z_0}{\delta_0}.
\]
Clearly, the polynomial $Q_0=z_0$ is of degree 1 and independent of $z_1,\dots, z_d$.

We apply Theorem~\ref{theorem:ortega-moreno2} to the homogeneous polynomials $\Qmod{j}{\delta_0}$ and the positive constants \(\delta_j\) restricted to the sphere $S^{2d+1}_{r(\delta_0)}$ of radius 
\[
r(\delta_0):=\sqrt{R^2+\delta_0^2}
\]
centered at the origin.

To do this, we consider the function $F_{\delta_0}$ %where $S^{2d+1}\subset \mathbb C \times \mathbb C^d$ is the sphere of radius $\sqrt{R^2+\delta_0^2}$, 
defined by 
\[
F_{\delta_0}=\left| \Qmod{0}{\delta_0}^{\delta_0^2} \Qmod{1}{\delta_0}^{\delta_1^2}\dots \Qmod{N}{\delta_0}^{\delta_N^2} \right|.
\]
%where $z=(z_1,\dots,z_d)\in \mathbb C^d$. 
As the polynomials $\Qmod{j}{\delta_0}$ are homogeneous, the function $F_{\delta_0}$ is well-defined on the complex projective space $S^{2d+1}_{r(\delta_0)}\big/S^1$, where $S^1\subset \mathbb C^1$ is the set of complex numbers of unit norm. Hence we may assume that the function $F_{\delta_0}$ is defined on the set
\[
\mathcal S_{\delta_0}=\left\{\,(t,z)\in \mathbb R_+\times \mathbb C^{d}\mid\, t^2+ |z|^2=R^2+\delta_0^2 \,\right\},
\]
here $\mathbb R_+$ is the set of non-negative reals. Remark that the restriction of $F_{\delta_0}$ to $\mathcal S_{\delta_0}$ is in fact a function depending only on $z\in \mathbb C^d$ as $t$ is well-defined if one knows the value of $z=(z_1,\dots, z_d)\in \mathbb C$, that is, 
\begin{equation}
\label{equation:t}
t=\sqrt{{\delta_0}^2+R^2-|z|^2}.
\end{equation}
Therefore, from now on, we assume that $F_{\delta_0}$ is a function defined on some subset of~$\mathbb C^d$.

To study the convergence of $F_{\delta_0}$ as $\delta_0\to+\infty$, we finally introduce the function $F: \mathbb C^{d}\to \mathbb R_+$ defined by
\[
F(z)=e^{\frac{R^2-|z|^2}{2}}\cdot \Big|  P_1^{\delta_1^2}(z)\dots P_N^{\delta_N^2}(z)\Big|,
\]
where $z\in \mathbb C^d$.

\begin{claim}
\label{claim:uniform}
The sequence of functions $F_{\delta_0}$ converges uniformly to $F$ on compact subsets of $\mathbb C^d$ as $\delta_0\to+\infty$.
\end{claim}

\begin{proof}
First, notice that the sequence of functions \[
\Qmod{0}{\delta_0}(t)=  \left(\frac{t}{\delta_0}\right)^{\delta_0^2}=\left(\frac{\sqrt{\delta_0^2+R^2-|z|^2}}{\delta_0}\right)^{\delta_0^2}=\left(1+\frac{R^2-|z|^2}{2\delta_0^2}+O(\delta_0^{-4})\right)^{\delta_0^2}
\]
converges on compact sets to $e^{\frac{R^2-|z|^2}{2}}$ because for $z$ from the compact subset of $\mathbb C^d$, there is an absolute constant for the term $O(\delta_0^{-4})$.

Second, notice that
\[
\left(\frac{t}{\delta_0}\right)^c =\left( 1+\frac{R^2-|z|^2}{2\delta_0^2}+O(\delta^{-4})\right)^c,
\]
where $c$ is some positive constant,
converges uniformly to 1 on a compact set as $\delta_0\to +\infty$. Since this expression with (probably different) constants $c$
appears finitely many times as a factor of monomials of the polynomial $\Qmod{j}{\delta_0}(t,z)$, we may conclude that $\Qmod{j}{\delta_0}(t,z)$ converges to $Q_j(1,z)=P_j(z)$ on a compact set of $\mathbb C^d$.

From that, we easily conclude the desired convergence.
\end{proof}

Since $e^{R^2}$ is a constant, considering the maximum of $F$ is the same as considering the maximum of 
\[
e^{-\frac{|z|^2}{2}} |P_1(z)|^{\delta_1^2}\cdots |P_N(z)|^{\delta_N^2}.
\]
from the statement of the theorem.

Take a sequence of $\delta_0$, $\{\delta_{0,n}\}$, tending to $+\infty$. Consider a point $\hat z_n = (t_n, z_n)$ of maximum of $F_{\delta_{0,n}}$. Theorem~\ref{theorem:ortega-moreno2} guarantees that $t \ge \delta_0$ and therefore $|z_n|\le R$. Passing to a subsequence, we may assume that the sequence of points $\{z_n\}$ also tends to a point $z_\infty\in B^{2d}(R)$. From Claim~\ref{claim:uniform} it follows that this point $z_\infty$ is a maximum of $F$ in the ball $B^{2d}(R)$, as we need.

It remains to establish inequalities on the distance between $z_\infty$ and each of the zero sets of $P_k$. Assume the contrary that $z_\infty$ is at distance strictly less than $\delta'_k < \delta_k$ from the zero set of $P_k$ in $\mathbb C^d$. It means that $z_n$ is at distance strictly less than $\delta'_k$ from the zero set of $P_k$ for sufficiently large $n$. Denote the corresponding zero point of $P_k$ by $w_n$, which is at distance strictly less than $\delta'_k$ from $z_n$. 

The point 
\[	
\widetilde w_n = \frac{r(\delta_{0,n})}{\sqrt{\delta_{0,n}^2 + |w_n|^2}} (\delta_{0,n}, w_n) = 
	\frac{\sqrt{1 + R^2/\delta_{0,n}^2}}{\sqrt{1 + \left( |w_n|^2 \right)/\delta_{0,n}^2}} (\delta_{0,n}, w_n)\in \mathcal S_{\delta_{0,n}} 
\]
is a zero point of $\Qmod{k}{\delta_{0,n}}$ on $\mathcal S_{\delta_{0,n}}$, because
\[
\Qmod{k}{\delta_{0,n}}(\widetilde w_n) = \frac{\left(R^2 + \delta_{0,n}^2\right)^{\deg P_k/2}}{\left(\delta_{0,n}^2 + |w_n|^2\right)^{\deg P_k/2}} \Qmod{k}{\delta_{0,n}} (\delta_{0,n}, w_n) = \frac{\left(R^2 + \delta_{0,n}^2\right)^{\deg P_k/2}}{\left(\delta_{0,n}^2 + |w_n|^2\right)^{\deg P_k/2}} P_k (w_n) = 0.
\]

Note that the factor here is $1 + O(\delta_{0,n}^{-2})$ and the distance of $\widetilde w_n$ from $(\delta_{0,n}, w_n)$ is therefore $O(\delta_{0,n}^{-1})\to 0$. 

Similarly, the point
\[
\widetilde z_n = \frac{r(\delta_{0,n})}{\sqrt{\delta_{0,n}^2 + |z_n|^2}} (\delta_{0,n}, z_n) =
\frac{\sqrt{1 + R^2/\delta_{0,n}^2}}{\sqrt{1 + \left( |z_n|^2 \right)/\delta_{0,n}^2}} (\delta_{0,n}, z_n)\in \mathcal S_{\delta_{0,n}}
\]
is at distance at most $O(\delta_{0,n}^{-1})\to 0$ from $(\delta_{0,n}, z_n)$. Also note that the point 
\[
\hat z_n = (t_n, z_n)\in \mathcal S_{\delta_{0,n}}
\]
is at distance $O(\delta_{0,n}^{-1})\to 0$ from $\widetilde z_n$, differing in the first coordinate only by at most $\sqrt{R^2+\delta_{0,n}^2} - \delta_{0,n} = O(\delta_{0,n}^{-1})$.

Since the distance between $(\delta_{0,n}, w_n)$ and $(\delta_{0,n}, z_n)$ is strictly less than $\delta'_k<\delta_k$ by our assumption, the distance between the points on the sphere $\mathcal S_{\delta_{0,n}}$, $\widetilde w_n$ and $\hat z_n$ is strictly less than $\delta_k$ for sufficiently large $n$. This contradicts the conclusion of Theorem~\ref{theorem:ortega-moreno2}.

\bibliography{../Bib/karasev}
\bibliographystyle{abbrv}

\end{document}